\documentclass[12pt]{article}  
\usepackage[T1]{fontenc}
\usepackage{amsxtra}

\usepackage{color}
\usepackage{amscd,amsfonts,amsmath,amssymb,
	amsthm,amstext,latexsym}
\usepackage{enumerate,pifont,graphicx}
\usepackage{xfrac}
\usepackage[english]{babel}
\usepackage[all,cmtip]{xy} 

\title{On Delaunay Ends in the DPW Method}
\author{Thomas Raujouan}

\newtheorem{proposition}{Proposition}
\newtheorem{theorem}{Theorem}
\newtheorem{lemma}{Lemma}
\newtheorem{definition}{Definition}
\newtheorem{corollary}{Corollary}
\newtheorem{remark}{Remark}
\newtheorem{hypotheses}{Hypotheses}

\newcommand{\la}{\lambda}
\newcommand{\Oo}{\mathcal{O}}
\newcommand{\Cc}{\mathbb{C}}

\newcommand{\Q}{\mathcal{Q}}

\newcommand{\Sym}{\mathrm{Sym}}
\newcommand{\D}{\mathcal{D}}
\newcommand{\Rr}{\mathbb{R}}
\newcommand{\Aa}{\mathcal{A}}
\newcommand{\I}{\mathrm{I}_2}
\newcommand{\Ss}{\mathbb{S}}
\newcommand{\Uni}{\mathrm{Uni}}
\newcommand{\R}{\mathcal{R}}
\newcommand{\Dd}{\mathbb{D}}
\newcommand{\M}{\mathcal{M}}
\newcommand{\Tub}{\mathrm{Tub}}
\newcommand{\N}{\mathcal{N}}

\newcommand{\Pos}{\mathrm{Pos}}
\newcommand{\Ll}{\mathcal{L}}
\newcommand{\Nn}{\mathbb{N}}

\newcommand{\SL}{\mathrm{SL}}

\newcommand{\SU}{\mathrm{SU}}
\newcommand{\su}{\mathfrak{su}}

\newcommand{\Res}{\mathrm{Res}}

\newcommand{\C}{\mathcal{C}}
\newcommand{\norm}[1]{\left\Vert#1\right\Vert}
\newcommand{\normR}[1]{\left\Vert#1\right\Vert_{\Rr^3}}
\newcommand{\conj}[1]{\overline{#1}}

\def\Im{{\rm Im}}
\newcommand{\E}{\mathcal{E}}
\newcommand{\Hh}{\mathbb{H}}

\newcommand{\LSL}{\Lambda  \mathrm{SL}_2\Cc}
\newcommand{\LrSL}[1]{\LSL}
\newcommand{\LSU}{\Lambda \mathrm{SU}_2}
\newcommand{\Lsu}{\Lambda \mathfrak{su}_2}
\newcommand{\Lslplus}{\Lambda_{+} \mathfrak{sl}_2\Cc}

\newcommand{\LrSLplus}[1]{\Lambda_+ \mathrm{SL}_2\Cc}
\newcommand{\LrSLplusR}[1]{\Lambda_+^\Rr \mathrm{SL}_2\Cc}
\newcommand{\Lrslplus}[1]{\Lambda_{+}\mathfrak{sl}_2\Cc}
\newcommand{\LrslplusR}[1]{\Lambda_+^{\Rr} \mathfrak{sl}_2\Cc}
\newcommand{\LrSU}[1]{\LSU}
\newcommand{\Lrsu}[1]{\Lambda \mathfrak{su}_2}
\newcommand{\Lrsl}[1]{\Lambda \mathfrak{sl}_2\Cc}

\newcommand{\LSLplusR}{\Lambda_+^\Rr \mathrm{SL}_2\Cc}

\newcommand{\function}[5]{
	\begin{array}{ccccc}
		#1 & : & #2 & \longrightarrow & #3 \\
		& & #4 & \longmapsto & #5 \\
	\end{array}
	}

\newcommand{\supp}[1]{\underset{#1}{\sup}}

\newcommand{\tendsto}[1]{\underset{#1}{\longrightarrow}}

\begin{document}

\maketitle

\begin{abstract}
	We consider constant mean curvature 1 surfaces in $\Rr^3$ arising via the DPW method from a holomorphic perturbation of the standard Delaunay potential on the punctured disk. Kilian, Rossman and Schmitt have proven that such a surface is asymptotic to a Delaunay surface. We consider families of such potentials parametrised by the necksize of the model Delaunay surface and prove the existence of a uniform disk on which the surfaces are close to the model Delaunay surface and are embedded in the unduloid case.
\end{abstract}

\section*{Introduction}

Beside the sphere, the simplest non-zero constant mean curvature (CMC) surface is the cylinder, which belongs to a one-parameter family of surfaces generated by the revolution of an elliptic function: the Delaunay surfaces, first described in \cite{delaunay}. Like the cylinder, Delaunay surfaces have two annular type ends, and Delaunay ends are the only possible embedded annular ends for a non-zero CMC surface. Indeed, as proven in \cite{korevaar} by Korevaar, Kusner and Solomon, if $\M\subset\Rr^3$ is a proper, embedded, non-zero CMC surface of finite topological type, then every annular end of $\M$ is asymptotic to a Delaunay surface and if $\M$ has exactly two ends which are of annular type, then $\M$ is a Delaunay surface. Thus, the status of Delaunay surfaces for non-zero CMC surfaces is very much alike the catenoid position in the study of minimal surfaces (see the result of Schoen in \cite{schoen}), and one has to understand the behaviour of Delaunay ends in order to construct examples of non-compact CMC surfaces with annular ends, as Kapouleas did in 1990 \cite{kapouleas}.

For an immersion, having a constant mean curvature and having a harmonic Gauss map are equivalent. This is why the Weierstrass type representation of Dorfmeister, Pedit and Wu \cite{dpw} has been used since the publication of their article to construct CMC surfaces. The DPW method can construct any conformal non-zero CMC immersion in $\Rr^3$, $\Hh^3$ or $\Ss^3$ with three ingredients: a holomorphic potential which takes its values in a loop group, a loop group factorisation, and a Sym-Bobenko formula. Several examples of CMC surfaces with annular ends, like $n$-noids and bubbletons, have been made by Dorfmeister, Wu, Killian, Kobayashi, McIntosh, Rossmann, Schmitt and Sterling \cite{dw,kkrs,kilian,kms,kobayashi,schmitt}. These constructions often rely on a holomorphic perturbation of the holomorphic potential giving rise to a Delaunay surface via the DPW method, and Kilian, Rossmann and Schmitt \cite{krs} have proven that such perturbations always induce asymptotically a Delaunay end.

More precisely, any Delaunay embedding can be obtained with a holomorphic potential of the form $\xi^\D=Az^{-1}dz$ where
\begin{equation*}
	A=\begin{pmatrix}
	0 & r\la^{-1} +s\\ r\la + s & 0
	\end{pmatrix}.
\end{equation*}
The main result of \cite{krs} states that any immersion obtained from a perturbed potential of the form ${\xi}=\xi^\D + \Oo(z^0)$ is asymptotic to an embedded half-Delaunay surface in a neighbourhood of $z=0$, provided that the monodromy problem is solved. In this paper, we allow the perturbed potential to move in the family of Delaunay potentials by introducing a real parameter $t$, proportional to the weight of the model Delaunay surface,  and consider $\xi_t=\xi_t^\D + \Oo_t(z^0)$ where $\xi_t^\D$ is a Delaunay potential of weight $8\pi t$. The main theorem of \cite{krs} states that for every $t>0$, there exists a small neighbourhood of the origin on which the surface produced by the potential $\xi_t$ is embedded and asymptotic to a half Delaunay surface. Unfortunately, without further hypotheses, this neighbourhood vanishes into a single point as $t$ tends to zero. Adding a few assumptions, we prove here that there exists a uniform neighbourhood of the origin upon which the surfaces induced by the family $\xi_t$ are all embedded and asymptotic to a half Delaunay surface for $t>0$ small enough.

Hence, the point of our paper is not to show that the ends of the perturbed unduloid family are embedded (which is what \cite{krs} does), but that all the immersions of this family are embedded on a uniform punctured disk. Equipped with our result, Martin Traizet (in \cite{nnoids} and \cite{minoids}) showed for the first time how the DPW method can be used to both construct CMC $n$-noids without symmetries and prove that they are Alexandrov embedded. 

The theorem we prove is the following one (definitions and notations are clarified in Section \ref{sectionDPW}):

\begin{theorem}
	\label{theoremsimplifie}
	Let $\Phi_t$ be a holomorphic frame arising from a perturbed Delaunay potential $\xi_t$ defined on a punctured neighbourhood of $z=0$. Suppose that $\Phi_0(1,\la) = \I$ and that the monodromy of $\Phi_t$ is unitary. Then, if $f_t$ denotes the immersion obtained via the DPW method,
	\begin{itemize}
		\item There exists a family $f_t^\D$ of Delaunay immersions such that for all $\alpha<1$ and $|t|$ small enough,
		\begin{equation*}
			\Vert f_t(z) - f_t^{\D}(z)\Vert _{\Rr^3} \leq C_\alpha |t||z|^{\alpha}
		\end{equation*}
		on a uniform neighbourhood of $z=0$.
		\item If $t>0$ is small enough, then $f_t$ is an embedding of a uniform neighbourhood of $z=0$.
		\item The limit axis of $f_t^\D$ as $t$ tends to $0$ can be made explicit.
	\end{itemize}
\end{theorem}

An outline of the proof is given in Section \ref{sectionoutline}, together with an explanation of why the convergence of $t$ to $0$ forbids us from using several key results of \cite{krs}.

\section{The DPW method}
\label{sectionDPW}

\subsection{Loop groups}
\label{sectionloopgroups}

Our maps will often depend on a spectral parameter $\la$ that can be in one of the following subsets of $\Cc$ ($R> 1$):
\begin{equation*}
	\begin{array}{ll}
	\D_R = \left\{ \la\in\Cc, \ |\la|<R \right\}, & \Aa_R = \left\{ \la\in \Cc, \ \frac{1}{R}<|\la|<R \right\} ,\\
	\D_1 = \left\{ \la\in\Cc, \ |\la|<1 \right\}, & \Aa_1 = \left\{ \la\in \Cc, \ |\la|=1 \right\} .
	\end{array}
\end{equation*}
For the coordinate $z$, we will note ($\epsilon >0$):
\begin{equation*}
	\begin{array}{ll}
	\Dd_\epsilon = \left\{ z\in\Cc, \ |z|<\epsilon \right\}, & \Ss_\epsilon = \left\{ z\in\Cc, \ |z|=\epsilon \right\}.
	\end{array}
\end{equation*}

Let us define the following (untwisted) loop groups and algebras:
\begin{itemize}
	\item $\LrSL{R}$ is the set of smooth maps $\Phi : \Aa_1\longrightarrow \SL_2\Cc$.
	\item $\LrSU{R}\subset \LSL$ is the set of maps $F\in \LSL$ such that $F(\la)\in\SU_2$ for all $\la\in\Aa_1$.
	\item $\LrSLplus{R}\subset\LrSL{R}$ is the set of maps $G\in\LrSL{R}$ that can be holomorphically extended to $\D_1$ and such that $G(0)$ is upper triangular.
	\item $\LrSLplusR{R} \subset\LrSLplus{R}$ is the set of maps $B\in\LrSLplus{R}$ such that $B(0)$ has positive elements on the diagonal.
	\item $\Lrsl{R}$ is the set of smooth maps $A : \Aa_1 \longrightarrow \mathfrak{sl}_2\Cc$.
	\item $\Lrsu{R}$ is the set of maps $m\in\Lrsl{R}$ such that $m(\la)\in\su_2$ for all $\la\in\Aa_1$.
	\item $\Lslplus\subset\Lrsl{R}$ is the set of maps $g\in \Lrsl{R}$ that can be holomorphically extended to $\D_1$ and such that $g(0)$ is upper triangular.
	\item $\LrslplusR{R} \subset \Lrslplus{R} $ is the set of maps $b\in \Lrslplus{R}$ such that $b(0)$ has real elements on the diagonal.
\end{itemize}
We also use the following notation:
\begin{equation*}
	\Oo(t^\alpha,z^\beta,\la^\gamma)=t^\alpha z^\beta \la^\gamma f(t,z,\la)
\end{equation*}
where $f$, on its domain of definition, is continuous with respect to $(t,z,\la)$ and holomorphic with respect to $(z,\la)$ for any $t$.
If one variable is not specified, its exponent is assumed to be $0$.

One step of the DPW method relies on the following Iwasawa decomposition (Theorem 8.1.1. in \cite{pressleysegal}): 

\begin{theorem}[Iwasawa decomposition]
	Any element $\Phi\in\LSL$ can be uniquely factorised into a product
	\begin{equation*}
		\Phi = F\times B
	\end{equation*}
	where $F\in\LSU$ and $B\in\LSLplusR$. Moreover, the map $\LSL \longrightarrow \LSU \times \LSLplusR$ is a $\C^\infty$ diffeomorphism for the intersection of the $\C^k$ topologies (see \cite{krs}).
\end{theorem}
The Iwasawa decomposition of a map $\Phi$ will often be written:
\begin{equation*}
	\Phi = \Uni\left(\Phi\right) \times \Pos\left(\Phi\right),
\end{equation*}
where $\Uni\left(\Phi\right)$ is called ``the unitary factor'' of $\Phi$ and $\Pos\left(\Phi\right)$ is ``the positive factor'' of $\Phi$. Using Corollary \ref{IwasawaExtended} of Appendix \ref{SectionIwaExtended}, note that if $\Phi$ is holomorphic on $\Aa_R$, then its unitary factor holomorphically extends to $\Aa_R$ and its positive factor holomorphically extends to $\D_R$.

\subsection{The $\mathfrak{su}_2$ model of $\Rr^3$}\label{su2model}

In the DPW method, immersions are given in a matrix model. The euclidean space $\Rr^3$ is thus identified with the Lie algebra $\mathfrak{su}_2$ by
\begin{equation*}
	x=(x_1,x_2,x_3) \simeq X=\frac{-i}{2}\begin{pmatrix}
	-x_3&x_1+ix_2\\x_1-ix_2&x_3
	\end{pmatrix}.
\end{equation*}
The canonical basis of $\Rr^3$ identified as $\su_2$ is denoted $\left(e_1,e_2,e_3\right)$.
In this model, the euclidean norm is given by
\begin{equation}\label{eqnormR3}
	\norm{x}^2 = 4 \det(X).
\end{equation}
Linear isometries are represented by the conjugacy action of $\SU_2$ on $\mathfrak{su}_2$:
\begin{equation*}
	H\cdot X = HXH^{-1}.
\end{equation*}

\subsection{The recipe}\label{recipe}

The DPW method takes for input data:
\begin{itemize}
	\item A Riemann surface $\Sigma$;
	\item A $\Lrsl{R}$-valued holomorphic 1-form $\xi=\xi(z,\la)$ on $\Sigma$ called ``the DPW potential'' which extends meromorphically to $\D_1$ with a pole only at $\la=0$, and which must be of the form
	\begin{equation*}
		\xi(z,\la) = \sum_{j=-1}^{\infty} \xi_j(z)\la^j
	\end{equation*}
	where each matrix $\xi_j(z)$ depends holomorphically on $z$ and all the entries of $\xi_{-1}(z)$ are zero except for the upper right entry which must never vanish;
	\item A base point $z_0\in\Sigma$;
	\item An initial condition $\Phi_{z_0}\in\LrSL{R}$.
\end{itemize}
Given such data, here are the three steps of the DPW method for constructing CMC-1 surfaces in $\Rr^3$ (in the untwisted setting):
\begin{enumerate}
	\item Solve for $\Phi$ the Cauchy problem with parameter $\la\in\Aa_1$:
	\begin{align*}
	\left\{
	\begin{array}{rcl}
	d_z\Phi(z,\la) &=& \Phi(z,\la)\xi(z,\la),\\
	\Phi(z_0,\la) &=& \Phi_{z_0}(\la).
	\end{array}
	\right.
	\end{align*}
	The solution $\Phi(z,\cdot)\in\LrSL{R}$ is called the ``holomorphic frame'' of the surface.
	In general, $\Phi(\cdot,\la)$ is only defined on the universal cover $\widetilde{\Sigma}$ of $\Sigma$ (see Section \ref{sectionmonod}). Note that if $\xi(z,\cdot)$ can be holomorphically extended to $\Aa_{R}$ ($R>1$), then $\Phi(z,\cdot)$ can also be holomorphically extended to $\Aa_{R}$ provided that $\Phi_{z_0}$ is holomorphic on $\Aa_R$.
	\item For all $z\in\widetilde{\Sigma}$, Iwasawa decompose $\Phi(z,\la) = F(z,\la) B(z,\la)$. The decomposition is done pointwise in $z$, but $F(z,\la)$ and $B(z,\la)$ depend real-analytically on $z$. The map $F$ is called the ``unitary frame'' of the surface.
	\item Define $f:\widetilde{\Sigma} \longrightarrow \su_2$ by the Sym-Bobenko formula:
	\begin{equation*}
		f(z)= \Sym(F) = i\frac{\partial F}{\partial \la}(z,1) F(z,1)^{-1} .
	\end{equation*}
	The map $f$ is then a conformal CMC-1 immersion whose normal map is given by
	\begin{equation}
	\label{defN}
	\N(z) = \frac{-i}{2}F(z,1)\begin{pmatrix}
	1&0\\0&-1
	\end{pmatrix}F(z,1)^{-1}.
	\end{equation}
	Its metric and Hopf differential are
	\begin{equation*}
		ds= 2\rho^2|\xi_{-1}^{12}||dz|,
	\end{equation*}
	\begin{equation*}
		Q=-2\xi_{-1}^{12}\xi_{0}^{21}dz^2
	\end{equation*}
	where $\xi_j^{kl}$ is the $(k,l)$-entry of the matrix $\xi_j(z)$ and $\rho$ is the upper-left entry of $B(z,0)$.
\end{enumerate}

The theory states that every conformal CMC-1 immersion can be obtained this way.

\subsection{Rigid motions of the surface}

Let $\xi$ be a DPW potential and $\Phi\in\LrSL{R}$ a solution of $d\Phi=\Phi\xi$. Take a loop $H\in\LrSU{R}$ that does not depend on $z$. Then $\widetilde{\Phi}=H\Phi$ also satisfies $d\widetilde{\Phi} = \widetilde{\Phi}\xi$ and gives rise to a rigid motion of the original surface given by $\Phi$. Let $f=\Sym \circ \Uni(\Phi)$ and $\widetilde{f} = \Sym\circ \Uni (\widetilde{\Phi})$. Then,
\begin{equation*}
	\widetilde{f}(z) = H(1)\cdot f(z) + \Sym(H).
\end{equation*}
This enjoins us to extend the action of section \ref{su2model} to affine isometries by
\begin{equation*}
H(\la)\cdot X = H(1)XH(1)^{-1} + i\frac{\partial H}{\partial \la}(1)H(1)^{-1}.
\end{equation*}
Note that $\LSU$ also acts on the tangent bundle of $\Rr^3$ via:
\begin{equation}
\label{actiontangent}
H\cdot \left(p,\vec{v}\right) = \left( H\cdot p, H(1)\cdot \vec{v} \right).
\end{equation}
This action will be useful to follow the axis of our surfaces: oriented affine lines are generated by pairs $(p,\vec{v})$ and the action of $\LSU$ on a given oriented affine line corresponds to the action \eqref{actiontangent} on its generators.

\subsection{Gauging}

Let $(\Sigma,\xi,z_0,\Phi_{z_0})$ be a set of DPW data with $d\Phi=\Phi\xi$. Let $G(z,\la)$ be a holomorphic map with respect to $z\in\Sigma$ such that $G(z,\cdot)\in\LrSLplus{R}$ (such a map is called an ``admissible gauge''). If we define $\widetilde{\Phi}=\Phi G$, then $\Phi$ and $\widetilde{\Phi}$ give rise to the same immersion $f$. This operation is called ``gauging'' and one can retrieve $\widetilde{\Phi}$ by applying the DPW method to the data $\left(\Sigma, \xi \cdot G,z_0,\Phi_{z_0}G(z_0,\cdot)\right)$ where
\begin{equation*}
	\xi \cdot G = G^{-1}\xi G + G^{-1}dG
\end{equation*}
is the action of gauges on potentials.

\subsection{The monodromy problem}\label{sectionmonod}

Since $\Phi$ is defined as the solution of a Cauchy problem on $\Sigma$, it is only defined on the universal cover $\widetilde{\Sigma}$ of $\Sigma$. For any deck transformation $\tau$ of $\widetilde{\Sigma}$, we define the monodromy matrix $\M_\tau\left(\Phi\right)\in\LrSL{R}$ as follow:
\begin{equation*}
	\Phi(\tau(z),\la) = \M_\tau(\Phi)(\la)\Phi(z,\la).
\end{equation*}
Note that $\M_\tau(\Phi)$ does not depend on $z$. The standard sufficient condition for the immersion $f$ to be be well-defined on $\Sigma$ is the following set of equations, called the monodromy problem in $\Rr^3$:
\begin{equation*}
\label{monodromyproblem}
\left\{
\begin{array}{rclc}
\M_\tau(\Phi)&\in&\LrSU{R},& (i)\\
\M_\tau(\Phi)(1) &=& \pm\I, & (ii)\\
\frac{\partial}{\partial \la}\M_\tau(\Phi)(1) &=& 0. &(iii)
\end{array}
\right.
\end{equation*}

\begin{remark}
	In this paper, the Riemann surface $\Sigma$ will always be a punctured neighbourhood $\Dd_\epsilon^*$ of $z=0$. Thus, all the deck transformations $\tau$ will be associated to a closed loop around $z=0$ and we will write $\M(\Phi)$ instead of $\M_\tau(\Phi)$.
\end{remark}

\begin{remark}
	\label{remarkmonodromy}
	Let $\Phi : \Cc^* \longrightarrow \LrSL{R}$ such that $\M\left(\Phi\right) \in \LrSU{R}$. Let $\widetilde{\Phi} = H \left(h^*\Phi\right) \cdot G$ where $H\in\LSL$, $G$ is holomorphic at $z=0$ and $h$ is a M\"obius transformation that leaves $z=0$ invariant. Then 
	\begin{equation*}
		\M(\widetilde{\Phi}) = H\M\left(\Phi\right)H^{-1}.
	\end{equation*}
	Thus, if the monodromy problem for $\Phi$ is solved, a sufficient condition for the monodromy problem for $\widetilde{\Phi}$ to be solved is that $H\in\LSU$.
\end{remark}

\subsection{The Delaunay family}\label{sectionExempleDelaunay}

Delaunay surfaces come in a one-parameter family: for all $t\in\left(-\infty, \frac{1}{16}\right]\backslash\left\{0\right\}$, there exists a unique Delaunay surface, whose weight (as defined in \cite{loopgroups}) is $8\pi t$. The DPW method can retrieve these surfaces using the following data:
\begin{equation*}
\begin{array}{cccc}
\Sigma = \Cc^*, & \xi_t(z,\la) = A_t(\la)z^{-1}dz, & z_0 = 1, & \Phi_{z_0}=\I,
\end{array}
\end{equation*}
where 
\begin{equation*}
A_t(\la) = \begin{pmatrix}
0 & r\la^{-1}+s\\
r\la+s & 0
\end{pmatrix}
\end{equation*}
and $r,s$ are functions of $t\in \left(-\infty, \frac{1}{16}\right]$ satisfying
\begin{equation}\label{conditionsrs}
\left\{ 
\begin{array}{c}
r,s\in\Rr,\\
r+s = \frac{1}{2},\\
rs = t.
\end{array}
\right.
\end{equation}
Note that the system \eqref{conditionsrs} admits two solutions, whether $r\geq s$ or $r\leq s$. For a fixed value of $t$, these two solutions give two different parametrisations of the same surface (up to a translation). If $r\geq s$, the unit circle of $\Cc^*$ is mapped onto a parallel circle of maximal radius: a bulge of the Delaunay surface. If $r\leq s$, the unit circle of $\Cc^*$ is mapped onto a parallel circle of minimal radius: a neck of the Delaunay surface. As $t$ tends to $0$ and in the case $r\geq s$, the immersions tend towards the parametrisation of a sphere on every compact subset of $\Cc^*$, which is why we call this setting the ``spherical case''. On the other hand, when $r\leq s$ and $t$ tends to $0$, the immersions degenerate into a point on every compact subset of $\Cc^*$. Nevertheless, we call this setting the ``catenoidal case'' because applying a blowup to the immersions makes them converge towards a catenoid on every compact subset of $\Cc^*$ (see \cite{minoids} for further details).

In any case, the corresponding holomorphic frame is explicit:
\begin{equation*}
\Phi_t(z,\la) = z^{A_t(\la)}
\end{equation*}
as is its monodromy around $z=0$:
\begin{equation}\label{eqmonodromy}
\M\left(\Phi_t\right)(\la) = \exp\left(2i\pi A_t(\la)\right)\\
= \cos\left( 2\pi\mu_t(\la) \right)\I + \frac{i\sin\left( 2\pi\mu_t(\la) \right)}{\mu_t(\la)}A_t(\la)	
\end{equation}
where
\begin{equation}\label{eqmut}
\mu_t(\la)^2 = -\det A_t(\la) = \frac{1}{4} + t \la^{-1}(\la-1)^2.
\end{equation}
Note that the conditions \eqref{conditionsrs} have been chosen in order for the monodromy problem of Section \ref{monodromyproblem} to be solved. The axis of the surface is given by $\left\{(x,0,-2r), \ x\in\Rr \right\}$ and its weight is $8\pi t$. Thus, the induced surface is an unduloid if $t>0$ and a nodoid if $t<0$.

\begin{remark}\label{remarkmubiendef}
	In order to deal with a single-valued square root of $\mu_t(\la)^2$ and to avoid some resonance cases in Section \ref{sectionzap}, we set $T>0$ and $R>1$ small enough for 
	\begin{equation*}
		\left\vert \mu_t(\la)^2 - \frac{1}{4} \right\vert < \frac{1}{4}
	\end{equation*}
	to hold for all $(t,\la)\in(-T,T)\times\Aa_R$.
\end{remark}

\subsection{Perturbed Delaunay DPW data}

We take a Delaunay potentials family as in section \ref{sectionExempleDelaunay} and we perturb it for $z$ in a small uniform neighbourhood of $0$:
\begin{definition}[Perturbed Delaunay potential]
	\label{defPerturbedDelaunay}
	Let $\epsilon>0$.
	A perturbed Delaunay potential is a one-parameter family $\{\xi_t\}_{t\in(-T,T)}$ of DPW potentials, holomorphic on $\Dd_\epsilon^*\times \Aa_R$ and of the form
	\begin{align*}
	\xi_t(z,\la) = A_t(\la)z^{-1}dz + R_t(z,\la)dz
	\end{align*}
	where $A_t$ is a Delaunay residue as in Section \ref{sectionExempleDelaunay}
	and $R_t(z,\la)\in\C^2$ with respect to $(t,z,\la)$, is holomorphic on $\Dd_\epsilon\times \Aa_R$ for all $t$ and satisfies $R_0(z,\la)=0$.
\end{definition}

The following set of hypotheses will be used to make sure that our holomorphic frames have a $\C^0$ regularity, are holomorphic with respect to $(z,\la)$, and solve the monodromy problem: 

\begin{hypotheses}\label{hypotheses}
	Let $\xi_t$ be a perturbed Delaunay potential. Let $\Phi_t$ be a holomorphic frame associated to it. We suppose that
	\begin{itemize}
		\item For some $t\in\left(-T,T\right)$ and $z\in \Dd_\epsilon^*$, $\Phi_t(z,\cdot)$ is holomorphic on $\Aa_R$,
		\item $\Phi_t(z,\la)$ is continous with respect to $(t,z,\la)$,
		\item The monodromy is unitary: $\M(\Phi_t)\in\LSU$.
	\end{itemize}
\end{hypotheses}

\begin{remark}
	When needed, one can replace $R>1$ by a smaller value in order for $\Phi_t$ to be holomorphic on $\Aa_R$ and continuous on $\overline{\Aa_R}$.
\end{remark}

The theorem we prove in this paper is the following:
\begin{theorem}
	\label{theorem}
	Let $\xi_t$ be a perturbed Delaunay potential and $\Phi_t$ a holomorphic frame associated to $\xi_t$ satisfying Hypotheses \ref{hypotheses} and such that $\Phi_0(1,\la) = \I$.
	Let $f_t=\Sym\left(\Uni(\Phi_t)\right)$. 
	Then,
	\begin{enumerate}
		\item For all $\alpha<1$ there exist constants $\epsilon>0$, $T>0$ and $C>0$ such that for all $0<|z|<\epsilon$ and $|t|<T$,
		\begin{align*}
		\Vert f_t(z) - f_t^{\D}(z)\Vert _{\Rr^3} \leq C|t||z|^{\alpha}
		\end{align*}
		where $f_t^\D$ is a Delaunay immersion of weight $8\pi t$.
		\item There exist $T'>0$ and $\epsilon'>0$ such that for all $0<t<T'$, $f_t$ is an embedding of $\left\{ 0<|z|<\epsilon' \right\}$.
		\item If $r\geq s$, the limit axis as $t$ tends to $0$ of $f_t^\D$ is the oriented line generated by $(-e_3,-\vec{e_1})$.
		
		If $r\leq s$, the limit axis as $t$ tends to $0$ of $f_t^\D$ is the oriented line generated by $(0,-\vec{e_1})$.
	\end{enumerate}
\end{theorem}

\begin{remark}
	\label{remarkPhi0defen1}
	We do not have to assume that $1\in \Dd_\epsilon$ for $\Phi_0$ to be defined at $z=1$. This only comes from the fact that $\xi_0$ is defined on $\Cc^*$, which implies that $\Phi_0$ is defined on the universal cover $\widetilde{\Cc^*}$.
\end{remark}

\subsection{Outline of the proof and comparison with \cite{krs}}\label{sectionoutline}

In Section \ref{sectionzap} we start the proof of Theorem \ref{theorem} by gauging the potential and changing coordinates. Starting from
\begin{equation*}
	\xi_t = A_t z^{-1}dz + \Oo(t,z^0)dz
\end{equation*}
we gain an order on $z$ and obtain the following new potential:
\begin{equation*}
	\widetilde{\xi}_t = A_tz^{-1}dz + \Oo(t,z)dz.
\end{equation*}
We then use the Fr\"obenius method and the new holomorphic frame is
\begin{equation*}
	\widetilde{\Phi}_t = \widetilde{M}_tz^{A_t}\left( \I + \Oo(t,z^2) \right).
\end{equation*}

In Section \ref{sectionconvergence}, we use this estimate on $\widetilde{\Phi}_t$ to prove the convergence of the immersions:
\begin{equation*}
	\norm{\widetilde{f}_t(z) - \widetilde{f}_t^\D(z)}_{\Rr^3} \leq C|t||z|^\alpha, \ \ \alpha<1
\end{equation*}
where $\widetilde{f}_t^\D$ is a Delaunay immersion whose axis can be explicitly computed. To do so, we need to know the asymptotic behaviour of the positive part $\Pos (\widetilde{\Phi}_t)$, which we compute using the fact that $\widetilde{f}_t^\D(\Cc^*)$ is a surface of revolution.

Finally, Section \ref{sectionembeddedness} proves that perturbations of unduloids are embedded on a uniform neighbourhood of the origin.

Although the method of this paper is inspired by what Kilian, Rossman and Schmitt did in \cite{krs}, their results cannot be used to prove our theorem. This is mainly because the asymptotics given in \cite{krs} for a fixed value of our parameter $t$ do not hold as $t$ tends to $0$. As an example, consider the proof of Lemma 2.5 in \cite{krs}: with our hypotheses, the constant they call $\kappa$ becomes a function of $t$ such that (with our notation of Section \ref{sectionpropertyxi})
\begin{equation*}
	\kappa\mid_{\substack{t=0}} = \frac{c_{12}(0,0)}{4} \neq 0.
\end{equation*}
Later in the proof, computing the determinant of the linear map $\Ll_1$ gives
\begin{equation*}
	\det \Ll_1 = \Oo(t)
\end{equation*}
and their gauged potential is then of the form
\begin{equation*}
	\widehat{\xi}_t = A_t z^{-1} dz + \Oo(t^{-1},z)dz,
\end{equation*}
the corresponding holomorphic frame being
\begin{equation*}
	\widehat{\Phi}_t = \widehat{M}_t z^{A_t} \left( \I + \Oo(t^{-1},z^2) \right).
\end{equation*}
Applying the Sym-Bobenko formula would give at best
\begin{equation}\label{eqconvkrs}
	\norm{\widehat{f}_t(z) - \widehat{f}_t^\D(z)}_{\Rr^3} \leq C\frac{1}{|t|}|z|^\alpha, \ \ \alpha<1
\end{equation}
which is not enough to show the convergence of the immersions on the compact sets of $\Cc^*$ as $t$ tends to $0$.
Note that gaining one order on $|t|$ in the estimate \eqref{eqconvkrs} is still not enough to show the embeddedness of $\widehat{f}_t$, since the first catenoidal neck of $\widehat{f}_t^\D$, which has a size of the order of $t$, is attained for $|z| \sim |t|$ as $t$ tends to $0$.

Finally, some bounds used in \cite{krs} such as (see Lemma 1.11 in \cite{krs})
\begin{equation*}
	c_1(\la) = \max_{x\in[0,\rho)}\norm{B(x,\la)}
\end{equation*}
depend on $t$ in our framework and may explode as $t$ tends to $0$.

\section{An application}
\label{sectiongauging}

Before proving Theorem \ref{theorem}, we must take account of the fact that one of its hypotheses is too restrictive. Indeed, $\Phi_0(1,\la)=\I$ has no reason to hold when one wants to construct examples, as Martin Traizet did in \cite{nnoids} and \cite{minoids}. We thus show here on a specific example how to ensure this hypothesis by gauging the potential and changing coordinates.

In all the section, $\xi_t$ is a perturbed Delaunay potential with $r\geq s$ 
and $\Phi_t$ is a holomorphic frame associated to $\xi_t$, satisfying Hypotheses \ref{hypotheses} and such that $\Phi_0(1,\la) = M(\la)$ where
\begin{equation}
\label{M}
M(\la) = \begin{pmatrix}
a&b\la^{-1}\\
c\la&d
\end{pmatrix}\in\LSL \ (a,b,c,d\in\Cc).
\end{equation}

After some simplification, we will be able to apply Theorem \ref{theorem} even though $\Phi_0(1,\la)\neq \I$. The only difference in the conclusion will be in the third point: the limit axis as $t$ tends to $0$ of the model Delaunay surface $f_t^\D$ will be the oriented line generated by $Q\cdot\left( 0, \vec{e_3} \right)$ where
\begin{equation}
\label{Q}
Q = \Uni\left[MH\right]
\end{equation}
with
\begin{equation}\label{defH}
H(\la) = \frac{1}{\sqrt{2}} \begin{pmatrix}
1 & -\la^{-1} \\
\la & 1
\end{pmatrix}.
\end{equation}
The method involves gauging, changing coordinates and applying an isometry, and relies on the fact that one can explicitly compute the Iwasawa decomposition of $MH$. Indeed, for all $a,b,c,d\in\Cc$ such that $ad-bc=1$,
\begin{equation}
\label{Iwaexplicit}
\begin{pmatrix}
a&b\la^{-1}\\c\la&d
\end{pmatrix}
= 
\frac{1}{\sqrt{|b|^2+|d|^2}} \begin{pmatrix}
\conj{d}&b\la^{-1}\\-\conj{b}\la&d
\end{pmatrix}
\times 
\frac{1}{\sqrt{|b|^2+|d|^2}} \begin{pmatrix}
1&0\\\left(a\conj{b} + c\conj{d}\right)\la&|b|^2+|d|^2
\end{pmatrix}
\end{equation}
is the Iwasawa decomposition of the left-hand side term. Note that if the matrix $M$ is explicit, then this formula makes both the matrix $Q$ in Equation \eqref{Q} and the limit axis of $f_t^\D$ explicit because $MH$ and $M$ have the same form.

\begin{lemma}\label{lemmacalculsjauge}
	Let $\xi_t$ be a perturbed Delaunay potential as in Definition \ref{defPerturbedDelaunay} with $r\geq s$. Let $\Phi_t$ be a holomorphic frame {associated to it}, satisfying Hypotheses \ref{hypotheses} and such that $\Phi_0(1,\la)=M(\la)$ as in \eqref{M}. Then there exists a M\"obius transformation that leaves $z=0$ invariant and a gauge $G$ such that:
	\begin{enumerate}
		\item the new potential $\widetilde{\xi}_t = (h^*\xi_t)\cdot G$ is also a perturbed Delaunay potential with the same residue than $\xi_t$,
		\item the holomorphic frame $\widetilde{\Phi}_t$ associated to $\widetilde{\xi}_t$ satisfies Hypotheses \ref{hypotheses} with $\widetilde{\Phi}_0(1,\la)\in\LSU$.
	\end{enumerate}
\end{lemma}
\begin{proof}
	Let $A_t$ and $R_t$ be as in Definition \ref{defPerturbedDelaunay}. Then
	\begin{equation*}
		\widetilde{\xi}_t = G^{-1}\left( A_th^{-1}dh + (h^*R_t)dh \right)G + G^{-1}dG.
	\end{equation*}
	The M\"obius transformation we are looking for satisfies $h(0)=0$ and thus
	\begin{equation*}
		h^{-1}dh = z^{-1}dz + \Oo(z)dz.
	\end{equation*}
	Wanting $\widetilde{\xi}_t$ to have a simple pole at $z=0$, we look for a gauge $G$ that is holomorphic at $z=0$. Wanting the residue of $\widetilde{\xi}_t$ to be $A_t$, we suppose that $G(0,\la)=\I$. These two conditions together with $\widetilde{\xi_0}=A_0z^{-1}dz$ enjoin us to solve the following Cauchy problem:
	\begin{equation}\label{pbcauchy}
		\left\{
		\begin{array}{rcl}
		dG &=& GA_0z^{-1}dz - A_0Gh^{-1}dh\\
		G(0) &=& \I.
		\end{array}
		\right.
	\end{equation}
	If we write
	\begin{equation*}
		h(z) = \frac{z}{pz + q}, \ p\in \Cc, \ q\in\Cc^*,
	\end{equation*}
	then the only solution of \eqref{pbcauchy} is given ({by Maple}) by: 
	\begin{equation*}\label{G}
		G(z,\la) = \begin{pmatrix}
		\sqrt{\frac{q}{pz+q}} & 0 \\
		\frac{\la pz}{\sqrt{q(pz+q)}} & \sqrt{\frac{pz+q}{q}}
		\end{pmatrix}
	\end{equation*}
	and a straightforward computation allows us to check that $G$ satisfies \eqref{pbcauchy}.
	Setting $0<\epsilon'<\epsilon$ with $\epsilon'<\frac{|q|}{|p|}$ if necessary, this proves the first point of the lemma.
	
	In order to prove the second point, diagonalise $A_0=HDH^{-1}$ with $H$ as in \eqref{defH} and compute
	\begin{equation}\label{phitilde01}
		\widetilde{\Phi}_0(1,\la) = M(\la)H(\la)\left( h(1)^D H(\la)^{-1}G(1,\la)H(\la) \right)H(\la)^{-1}
	\end{equation}
	where 
	\begin{equation*}
		D = \begin{pmatrix}
		\frac{1}{2} & 0\\0 & \frac{-1}{2}
		\end{pmatrix}.
	\end{equation*}
	Hence $\widetilde{\Phi}_0(1,\cdot)$ is holomorphic on $\Aa_R$. Moreover, the fact that $\widetilde{\xi}_t$ is $\C^2$ in $(t,z,\la)$ together with remark \ref{remarkmonodromy} imply that $\widetilde{\Phi}_t$ satisfies Hypotheses \ref{hypotheses}. Finally, compute
	\begin{equation*}
		h(1)^D H(\la)^{-1}G(1,\la)H(\la) = \begin{pmatrix}
		\frac{1}{\sqrt{q}} & 0 \\ \la \frac{p}{\sqrt{q}} & \sqrt{q}
		\end{pmatrix}
	\end{equation*}
	and, using Equation \eqref{Iwaexplicit},
	\begin{equation*}
		\Pos \left(MH\right) = \begin{pmatrix}
		\rho & 0\\
		\la\mu & \rho^{-1}
		\end{pmatrix}
	\end{equation*}
	where
	\begin{equation*}
		\rho = \frac{\sqrt{2}}{\sqrt{|b-a|^2 + |d-c|^2}}, \qquad \mu =\frac{1}{\sqrt{2}}\times  \frac{(a+b)(\bar{b} - \bar{a})+(c+d)(\bar{d}-\bar{c})}{\sqrt{|b-a|^2+|d-c|^2}}.
	\end{equation*}
	Then, setting
	\begin{equation*}
		\begin{array}{cc}
		p = -\rho \mu, & q=\rho^2,
		\end{array}
	\end{equation*}
	Equation \eqref{phitilde01} becomes ($Q$ is defined in \eqref{Q})
	\begin{equation*}
		\widetilde{\Phi}_0(1,\la) = Q H^{-1}\in\LSU
	\end{equation*}
	because $H\in \LSU$.
\end{proof}

If one wants to apply Theorem \ref{theorem}, it then suffices to set
\begin{equation*}
	\widehat{\Phi}_t = HQ^{-1}\widetilde{\Phi}_t
\end{equation*}
where $\widetilde{\Phi}_t$ is constructed by Lemma \ref{lemmacalculsjauge}. Let $\widehat{f}_t^\D$ be the model Delaunay immersion towards which the immmersion $ \Sym \left(\Uni (\widehat{\Phi}_t)\right)$ converges.
Theorem \ref{theorem} then states that the limit axis as $t$ tends to $0$ of $\widehat{f}_t^\D$ is the oriented line generated by $(-e_3,-\vec{e_1})$. Compute
\begin{equation*}
	H^{-1}\cdot (-e_3,-\vec{e_1}) = (-e_3, \vec{e_3})\simeq (0,\vec{e_3})
\end{equation*}
to prove that $\Sym\left( \Uni (\Phi_t) \right)$ converges to a model Delaunay surface whose limit axis as $t$ tends to $0$ is $Q\cdot (0,\vec{e_3})$. The following corollary summarises this section:
\begin{corollary}
	Let $\xi_t$ be a perturbed Delaunay potential with $r\geq s$ and $\Phi_t$ a holomorphic frame associated to $\xi_t$ satisfying Hypotheses \ref{hypotheses} and such that $\Phi_0(1,\la)$ is of the form given by \eqref{M}.
	Let $f_t=\Sym\left(\Uni(\Phi_t)\right)$. 
	Then,
	\begin{enumerate}
		\item For all $\alpha<1$ there exist constants $\epsilon>0$, $T>0$ and $C>0$ such that for all $0<|z|<\epsilon$ and $|t|<T$,
		\begin{align*}
		\Vert f_t(z) - f_t^{\D}(z)\Vert _{\Rr^3} \leq C|t||z|^{\alpha}
		\end{align*}
		where $f_t^\D$ is a Delaunay immersion of weight $8\pi t$.
		\item There exist $T'>0$ and $\epsilon'>0$ such that for all $0<t<T'$, $f_t$ is an embedding of $\left\{ 0<|z|<\epsilon' \right\}$.
		\item The limit axis as $t$ tends to $0$ of $f_t^\D$ is the oriented line generated by $Q\cdot(0,\vec{e_3})$ where $Q$ is given by Equation \eqref{Q}.
	\end{enumerate}
\end{corollary}

\section{The $z^AP$ form of $\Phi_t$}
\label{sectionzap}

Let us start the proof of Theorem \ref{theorem}: let $\xi_t$ be a perturbed Delaunay potential and $\Phi_t$ a holomorphic frame associated to $\xi_t$ satisfying Hypotheses \ref{hypotheses} and such that $\Phi_0(1,\la) = \I$.

In this section, we want to apply the Fr\"obenius method and write $\Phi_t$ in a $z^AP$ form. Unfortunately, the underlying Fuchsian system seems to admit resonance points. Our goal is to avoid them and to gain an order of convergence in the matrix $P$ of the $z^AP$ form. We will obtain the following result:

\begin{proposition}
	\label{propzAP}
	There exist a change of coordinate $h_t$ and a gauge $G_t$ such that, denoting 
	\begin{equation*}
	\widetilde{\Phi}_t=h_t^*\left(\Phi_tG_t\right)
	\end{equation*}
	and 
	\begin{equation*}
	\widetilde{\xi}_t=h_t^*\left(\xi_t\cdot G_t\right),
	\end{equation*}
	$\widetilde{\xi}_t$ is a perturbed Delaunay potential and $\widetilde{\Phi}_t$ is a holomorphic frame associated to $\widetilde{\xi}_t$ satisfying Hypotheses \ref{hypotheses} and such that $\widetilde{\Phi}_0(1,\la) = \I$. Moreover,
	\begin{equation}
	\label{eqzAPordre2}
	\widetilde{\Phi}_t(z,\la) = \widetilde{M}_t(\la)z^{A_t(\la)}\widetilde{P}_t(z,\la)
	\end{equation}
	where $\widetilde{M}_t\in \LSL$ is continuous and holomorphic on $\Aa_R$ for all $t$ and $\widetilde{P}_t : \Dd_{\epsilon'}\longrightarrow \LSL$ is $\C^2$, holomorphic on $\Dd_\epsilon'\times \Aa_R$ for all $t$ and satisfies $\widetilde{P}_t(z,\la) = \I+\Oo(t,z^2)$.
\end{proposition}

\subsection{Extending to the resonance points}

In this section, we use the Fr\"obenius method to write $\Phi_t$ in a $z^AP$ form, and extend this form to the resonance points. We will thus prove:
\begin{proposition}
	\label{propzap}
	There exist $M_t\in\LSL$ continuous and holomorphic on $\Aa_R$ for all $t$ and ${P}_t : \Dd_\epsilon\longrightarrow \LSL$ continuous and holomorphic on $\Dd_\epsilon\times \Aa_R$ for all $t$ satisfying ${P}_t(0,\la) = \I$ and
	\begin{equation*}
	{\Phi}_t(z,\la) = {M}_t(\la)z^{A_t(\la)}{P}_t(z,\la).
	\end{equation*}
\end{proposition}

Let us first recall the Fr\"obenius method in the non-resonant case (see \cite{teschl} and \cite{taylor}).
Let $\epsilon>0$ and $\xi$ be a holomorphic $1$-form from $\Dd_\epsilon^*$ to $\M_2(\Cc)$ defined by
\begin{equation*}
	\xi(z) = Az^{-1}dz + \sum_{k\in\Nn}C_kz^kdz.
\end{equation*}
For all $k\in\Nn$, let $P_k$ solve
\begin{equation}\label{eqrecfrobenius}
	\left\{
	\begin{array}{rcl}
	P_0&=&\I,\\
	\Ll_{k+1}(P_{k+1}) &=& \sum\limits_{i+j=k}P_iC_j
	\end{array}
	\right.
\end{equation}
where for all $n\in\Nn$,
\begin{equation*}\label{Ll}
	\function{\Ll_n}{\M_2(\Cc)}{\M_2(\Cc)}{X}{\left[ A,X \right]+nX.}
\end{equation*}
Then $P(z) = \sum_{k\in\Nn}P_kz^k$ is holomorphic on $\Dd_\epsilon$ and $\Phi(z) = z^AP(z)$ is holomorphic on the universal cover $\widetilde{\Dd_\epsilon^*}$ of $\Dd_\epsilon^*$ and solves $d\Phi=\Phi\xi$.

Let us now recall Lemma 2.2 of \cite{krs} in our framework:
\begin{lemma}
	Let $A\in\mathfrak{sl}_2\Cc$ such that $A^2 = \mu^2\I$. Then for all $n\in\Nn$,
	\begin{equation}\label{detL}
		\det \Ll_n = n^2 \left(n^2-4\mu^2\right)
	\end{equation}
	and
	\begin{equation}\label{inverseL}
		\Ll_n^{-1}\left(X\right) = \frac{1}{n}\left( X - \frac{1}{n^2-4\mu^2}\left( n\I - 2A \right)\left[ A,X \right] \right)	
	\end{equation}
\end{lemma}

Corollary \ref{corollaryLninvertible} follows from Remark \ref{remarkmubiendef} and Equation \eqref{detL}.

\begin{corollary}\label{corollaryLninvertible}
	Let $\Ll_{t,n}(X) = \left[ A_t(\la),X \right] + nX$. 
	\begin{itemize}
		\item For all $n\geq 2$, $\Ll_{t,n}$ is invertible on $(t,\la)\in\left(-T,T\right)\times \D_R^*$.
		\item For $n=1$, $\Ll_{t,1}$ is invertible on $(t,\la)\in \left(-T,T\right) \backslash \{0\} \times \D_R^*\backslash\{1\}$.
	\end{itemize}
\end{corollary}

\begin{remark}
	\label{remarkP}
	If we use the Ansatz given by the Fr\"obenius method and write
	\begin{equation}
	\label{Phizap}
	\Phi_t(z,\la) = M_t(\la)z^{A_t(\la)}P_t(z,\la)
	\end{equation}
	where
	\begin{equation*}
	\label{Pt}
	P_t(z,\la) = \sum\limits_{k=0}^{\infty} P_{t,k}(\la)z^k,
	\end{equation*}
	note that the resonance points only occur in the computation of $P_{t,1}(\la)$ because $\Ll_{t,n}$ is invertible on $(t,\la)\in\left(-T,T\right)\times \Aa_R$ for all $n\geq 2$. Thus, we only need to extend $P_{t,1}(\la)$ at $t=0$ and $\la=1$ to extend the $z^AP$ form of $\Phi_t$. According to \eqref{eqrecfrobenius}, 
	\begin{equation}
	\label{Pt1}
	P_{t,1}(\la) = \Ll_{t,1}^{-1}(tC_t(\la))
	\end{equation}
	and the form of $\det \Ll_{t,1}$ shows that $P_{t,1}$ has at most a pole of order $2$ at $\la=1$. Moreover, $\det \Ll_{t,1} = \Oo(t)$ and $tC_t=\Oo(t)$, so we already know that $P_t$ (and as a consequence, $M_t$) extends to $t=0$.
\end{remark}

It remains to extend the $z^AP$ form \eqref{Phizap} to $\la=1$. To do this, we adapt the techniques used in Lemma 2.5 of \cite{schmitt} to prove the following unitary $\times$ commutator lemma:

\begin{lemma}
	\label{lemmeunitarycommutator}
	Let $M:\Aa_R\backslash\{1\}\longrightarrow\SL_2\Cc$ holomorphic on $\Aa_R\backslash\{1\}$ with at most a pole at $\la=1$. Let $t\neq0$, $\Q=\exp\left(2i\pi A_t\right)\in\LSU$ and suppose that for all $\la \in \Aa_1\backslash\{1\}$, $M\Q M^{-1}\in \SU_2$.
	Then there exist $U\in\LSU$ and $K:\Aa_{R}\backslash\{1\}\longrightarrow\SL_2\Cc$ holomorphic such that
	\begin{equation*}
	\left\{
	\begin{array}{c}
	M=UK\\
	\left[A_t,K\right]=0.
	\end{array}
	\right.
	\end{equation*}
\end{lemma}
\begin{proof}
	We first apply Lemma 2.5 of \cite{schmitt} to construct $U$ and $K$ satisfying $M=UK$ and $\left[\Q,K\right]=0$ on $\Aa_1\backslash\{1\}$. The map $U$ is holomorphic on a small neighbourhood of $\Aa_1$. Without loss of generality, let this neighbourhood be $\Aa_R$. Then, $K$ is meromorphic on $\Aa_R\backslash\{1\}$ with at most a pole  at $\la=1$. Hence the map $\la \longmapsto \left[\Q(\la),K(\la)\right]$ is holomorphic on $\Aa_R\backslash\{1\}$ and vanishes on $\Aa_1\backslash\{1\}$. Thus, for all $\la\in\Aa_R\backslash\{1\}$,
	\begin{equation}
	\label{QK}
	\left[\Q(\la),K(\la)\right]=0.
	\end{equation}
	Recalling Equation \eqref{eqmonodromy},
	\begin{equation*}
	\Q=\cos(2\pi\mu_t)\I + \frac{i \sin(2\pi\mu_t)}{\mu_t}A_t.
	\end{equation*}
	Hence Equation \eqref{QK} implies that $\left[A_t,K\right]=0$ wherever $\mu_t(\la)^2 \neq \frac{1}{4}$. Using \eqref{eqmut}, $\left[A_t(\la),K(\la)\right]=0$ for all $(t,\la) \in\left(-T,T\right) \backslash \{0\}\times \Aa_R\backslash\{1\}$.
\end{proof}

We can now extend the  $z^AP$ form of $\Phi_t$ to $\la=1$. For $t\neq 0$ and $\la\in\Aa_1\backslash\{1\}$, use Lemma \ref{lemmeunitarycommutator} to write
\begin{equation*}
\Phi_t(z,\la) = U_t(\la)z^{A_t(\la)}K_t(\la)P_t(z,\la).
\end{equation*}
Let $\epsilon>0$ small enough for $P_t(\cdot, \la)$ to be defined on $\overline{\Dd}_\epsilon$. On $\Ss_\epsilon\times\Aa_1\backslash\{1\}$, $\Phi_t$ and $z^{A_t}$ are bounded. Then the map $(z,\la)\longmapsto K_tP_t$ is bounded on $\Ss_\epsilon\times\Aa_1\backslash\{1\}$ and holomorphic on $\Dd_\epsilon\times\Aa_1\backslash\{1\}$, so it is bounded on $\Dd_\epsilon\times\Aa_1\backslash\{1\}$. But $P_t(0,\la)=\I$, so $K_t$ is bounded on $\Aa_1\backslash\{1\}$. Thus, $P_t$ is bounded on $\Dd_\epsilon\times\Aa_1\backslash\{1\}$. But $P_t$ is holomorphic on $\Dd_\epsilon\times \Aa_R\backslash\{1\}$ with at most a pole at $\la=1$, so $P_t$ is holomorphic on $\Dd_\epsilon\times\Aa_R$ and $M_t$ is holomorphic on $\Aa_R$. This ends the proof of Proposition \ref{propzap}.

\subsection{A property of $\xi_t$}\label{sectionpropertyxi}

The fact that there exists a holomorphic frame $\Phi_t$ associated to $\xi_t$ such that $\M\left( \Phi_t \right)\in\LSU$ and $\Phi_0(1,\la) = \I$ gives us a piece of information on the potential $\xi_t$.
Let $C_t(\la)\in\mathfrak{sl}_2\Cc$ so that
\begin{equation*}
	\xi_t(z,\la) = A_t(\la)z^{-1}dz + tC_t(\la)dz + \Oo(t,z)dz
\end{equation*}
and write
\begin{equation}\label{defCt}
	C_t(\la) = \begin{pmatrix}
	c_{11}(t,\la) & \la^{-1}c_{12}(t,\la)\\
	c_{21}(t,\la) & -c_{11}(t,\la)
	\end{pmatrix}.
\end{equation}
Define
\begin{equation}\label{defpt}
	p_t = \frac{sc_{12}(t,0) + rc_{21}(t,0)}{2}.
\end{equation}
\begin{lemma}\label{lemmep0}
	The quantity $p_t$ vanishes at $t=0$.
\end{lemma}
\begin{proof}
	First, note that $\Phi_0(1,\la) = \I$ implies that $\Phi_0(z,\la) = z^{A_0(\la)}$, and thus $\M(\Phi_0)=-\I$. Let $\gamma\subset \Dd_\epsilon^*$ be a closed loop around $0$. Apply Proposition \ref{propderiveemonod} of Appendix \ref{appendixderiveemonod} to get ($X'$ denotes the derivative of $X$ at $t=0$ and $R_t$ is the holomorphic part of $\xi_t$)
	\begin{align*}
	\M(\Phi_t)' &=  \int_{\gamma} z^{A_0}\xi'z^{-A_0} \times \M(\Phi_0)\\
	&= -\int_{\gamma}z^{A_0}\left(A'z^{-1}\right)z^{-A_0}dz - \int_{\gamma} z^{A_0} R' z^{-A_0}dz \\
	&= \M(z^{A_t})' -  \int_{\gamma} z^{A_0}R'z^{-A_0}dz.
	\end{align*}
	But $\M(\Phi_t),\M(z^{A_t})\in\LSU$ and $\M(\Phi_0)=\M(z^{A_0}) = -\I$. Thus, $\M(\Phi_t)',\M(z^{A_t})'\in\Lsu$ and
	\begin{equation}\label{eqintegralsu2}
		\int_{\gamma} z^{A_0}R'z^{-A_0}dz \in\Lsu.
	\end{equation}
	Diagonalise $A_0 = H D H^{-1}$ with
	\begin{equation*}
	D = \begin{pmatrix}
	\frac{1}{2} & 0\\
	0 & \frac{-1}{2}
	\end{pmatrix}
	\end{equation*}
	and $H\in\LSU$ to be expressed later.
	Then
	\begin{equation*}
		z^D = \frac{1}{\sqrt{z}}\begin{pmatrix}
		z & 0\\0&1
		\end{pmatrix}
	\end{equation*}
	and
	\begin{align*}
	\int_{\gamma} z^{A_0}R'z^{-A_0} dz &= \int_{\gamma} Hz^DH^{-1}\left( C_0 + \Oo(z) \right)Hz^{-D}H^{-1}\\
	&= H\left( \Res_{z=0} z^DH^{-1}C_0Hz^{-D} \right)H^{-1}.
	\end{align*}
	Equation \eqref{eqintegralsu2} and $H\in\LSU$ imply that 
	\begin{equation}\label{eqResinsu2}
		\Res_{z=0} \left(z^DH^{-1}C_0Hz^{-D}\right)\in\Lsu.
	\end{equation}
	Denoting by $c(\la)$ the bottom-left entry of $H^{-1}C_0H$ and looking at the product $z^D(H^{-1}C_0H)z^{-D}$, Equation \eqref{eqResinsu2} gives
	\begin{equation*}
		\begin{pmatrix}
		0 & 0 \\
		c(\la) & 0
		\end{pmatrix}\in\Lsu
	\end{equation*}
	and thus,
	\begin{equation}\label{c=0}
		c(\la) = \left( H^{-1}C_0H \right)_{21} \equiv 0
	\end{equation}
	Two cases can occur:
	\begin{itemize}
		\item If $r\geq s$, 
		\begin{equation*}
		H = \frac{1}{\sqrt{2}}\begin{pmatrix}
		1 & -\la^{-1} \\
		\la & 1
		\end{pmatrix}\in\LSU
		\end{equation*}
		and computation gives
		\begin{equation*}\label{defclambda}
			c(\la) = -\la\left( c_{11}(0,\la) + \frac{c_{12}(0,\la)}{2} \right)+\frac{c_{21}(0,\la)}{2}.
		\end{equation*}
		Using Equation \eqref{c=0}, $c_{21}(0,0)=0$ and $p_0=0$.
		\item If $r\leq s$, the same reasoning applies with
		\begin{equation*}
			H(\la) = \frac{1}{\sqrt{2}}\begin{pmatrix}
			1 & -1 \\
			1 & 1
			\end{pmatrix}\text{ and }c(\la) = -\la^{-1}\frac{c_{12}(0,\la)}{2} + \frac{c_{21}(0,\la)}{2} - c_{11}(0,\la).
		\end{equation*}
		Thus, $c_{12}(0,0)=0$ and $p_0=0$.
	\end{itemize}
\end{proof}

\subsection{Gaining an order of convergence}

We can now prove Proposition \ref{propzAP} by following the method used in Section 2.2 of \cite{krs}: gauging the potential. The gauge we will use is of the following form:
\begin{equation}
\label{Gt}
G_t(z,\la) = \exp\left( g_t(\la)z \right)
\end{equation}
which is an admissible gauge provided that $g_t\in\Lslplus$. This is why we need the following lemma:

\begin{lemma}
	\label{lemmagauge}
	Let 
	\begin{equation*}
	\label{gt}
	g_t(\la) = p_tA_t(\la) - P_{t,1}(\la)
	\end{equation*}
	where $P_{t,1}$ is defined in Equation \eqref{Pt1}. Then
	\begin{enumerate}
		\item The map $g_t$ is in $\Lslplus$.
		\item The map $g_t$ extends to $t=0$ with $g_0=0$.
	\end{enumerate}
\end{lemma}
\begin{proof}
	To prove the first point, let $t\neq 0$ and use Equations \eqref{Pt1}, \eqref{defCt}, \eqref{inverseL} and \eqref{defpt} to compute (this is a tedious calculation)
	\begin{align*}
	P_{t,1}(\la) = \la^{-1} \begin{pmatrix}
	0 & rp_t\\
	0&0
	\end{pmatrix} + \la^0 \begin{pmatrix}
	\star  & \star \\
	sp_t & \star
	\end{pmatrix} +  \Oo(\la).
	\end{align*}
	Thus,
	\begin{equation*}
		g_t(\la) = p_tA_t(\la) - P_{t,1}(\la) = \la^{-1} \begin{pmatrix}
		0 & 0\\
		0&0
		\end{pmatrix} + \la^0 \begin{pmatrix}
		\star  & \star \\
		0 & \star
		\end{pmatrix} +  \Oo(\la).
	\end{equation*}
	
	For the second point, use Equations \eqref{Pt1} and \eqref{inverseL} to write for $t\neq 0$:
	\begin{equation*}
		P_{t,1} = t\Ll_{t,1}^{-1}\left( C_t \right) = t\left( C_t - \frac{1}{1-4\mu_t^2}\left( \I -2A_t \right)\left[ A_t,C_t \right] \right).
	\end{equation*}
	Note that $C_t$ is continuous at $t=0$ because $\xi_t\in\C^2$ and that $1-4\mu_t^2 = \Oo(t)$ to extend $P_{t,1}$ to $t=0$. Moreover, recall Lemma \ref{lemmep0}, Equation \eqref{eqmut} and diagonalise $A_0=HDH^{-1}$ to get:
	\begin{equation*}
		g_0 = \frac{-\la}{4(\la-1)^2}H\left( \I - 2D \right)\left[ D,H^{-1}C_0H \right]H^{-1}.
	\end{equation*}
	A straightforward computation gives
	\begin{equation*}
		\left( \I - 2D \right)\left[ D,H^{-1}C_0H \right] = \begin{pmatrix}
		0&0\\-2c(\la) & 0
		\end{pmatrix}
	\end{equation*}
	with $c(\la)$ as in Equation \eqref{c=0}. Hence $g_0=0$.
\end{proof}

Let $G_t$ be the gauge defined by \eqref{Gt}. Then the gauged potential has the form
\begin{align*}
\xi_t\cdot G_t (z,\la) &= A_t(\la)z^{-1}dz + \left(\left [A_t(\la),g_t(\la)\right]+g_t(\la) + tC_t(\la)  \right)dz + \Oo(t,z)dz +\Oo(g_t^2z)dz\\
&= A_t(\la)z^{-1}dz + \left(\Ll_{t,1}(g_t(\la)) + tC_t(\la)  \right)dz + \Oo(t,z)dz\\ 
&= A_t(\la)z^{-1}dz + p_tA_t(\la)dz + \Oo(t,z)dz,
\end{align*}
because of Equation \eqref{Pt1}. This gauge has been chosen to fit with the following change of coordinate:
\begin{equation*}
\label{ht}
h_t(z) = \frac{z}{1+p_tz}.
\end{equation*}
The resulting potential (defined in Proposition \ref{propzAP}) is then
\begin{equation*}
\label{xitilde1}
\widetilde{\xi}_t = A_t\frac{dz}{1+p_tz} + p_tA_t\frac{dz}{(1+p_tz)^2} + \Oo(t,z)dz = A_tz^{-1}dz +\Oo(t,z)dz
\end{equation*}
because $p_0=0$. Apply the Fr\"obenius method to $\widetilde{\xi}_t$ to obtain \eqref{eqzAPordre2} and choose $\epsilon'\leq \epsilon$ such that for all $t\neq 0$, $\epsilon'<|p_t|^{-1}$
to end the proof of Proposition \ref{propzAP}.

\section{Convergence of immersions}
\label{sectionconvergence}

In this section, we prove the first and third points of Theorem \ref{theorem}.
In the end, we want to compare $\Phi_t(z,\la) = M_t(\la)z^{A_t(\la)}\left(\I + \Oo(t,z^2)\right)$ to 
\begin{equation*}
\label{PhitD}
\Phi_t^\D(z,\la) = M_t(\la)z^{A_t(\la)}.
\end{equation*}
We will denote
\begin{equation*}
\label{FtD}
F_t^\D = \Uni(\Phi_t^D)
\end{equation*}
and
\begin{equation*}
\label{ftD}
f_t^\D=\Sym (F_t^\D).
\end{equation*}
We first want to make sure that $\Phi_t^\D$ induces a Delaunay surface for all $t$. 
For this purpose, recall Lemma 1.12 in \cite{krs}, which implies that $f_t^\D$ is a Delaunay surface of weight $8\pi t$. Hence, there exists a rigid motion $\phi$ of $\Rr^3$ such that $\phi\circ f_t^\D$ has the following parametrisation:
\begin{equation*}
\begin{array}{ccccc}
\phi\circ f_t^\D & : & \Sigma & \longrightarrow & \Rr^3 \\
& & z=e^{x+iy} & \longmapsto & \left( \tau_t(x), \sigma_t(x)\cos y, \sigma_t(x)\sin y \right) \\
\end{array}
\end{equation*}
where $\left( \tau_t(x), \sigma_t(x) \right)$ is the profile curve of the surface. Recalling that the coordinates are isothermal gives the following metric:
\begin{equation}\label{eqmetric}
	ds_t^2 = \sigma_t^2 \frac{|dz|^2}{|z|^2}.
\end{equation}
Let us compare the asymptotic behaviours of the unitary parts of $\Phi_t$ and $\Phi_t^\D$ for $\la \in \Aa_1$ using, as in \cite{krs}, a Cauchy formula. We will use the following norms:
\begin{itemize}
	\item For $v=(v_1,v_2)\in\Cc^2$, $|v|=\left( |v_1|^2+|v_2|^2 \right)^{\frac{1}{2}}$.
	\item For $M\in\M_2(\Cc)$, $\norm{M} = \underset{|v|=1}{\sup}\left| Mv \right|$.
	\item For $\Psi :\E \longrightarrow \M_2(\Cc)$, $\norm{\Phi}_{\E} = \underset{\la\in\E}{\sup}\norm{\Psi(\la)}$.
\end{itemize}

\begin{lemma}
	\label{asymptoticF}
	For all $\alpha<1$ there exist constants $\epsilon>0$, $T>0$ and $C>0$ such that for all $0<|z|<\epsilon$ and $|t|<T$,
	\begin{equation}
	\label{Cauchy0}
	\norm{\left(F_t^\D\right)^{-1}F_t-\I}_{\Aa_1} \leq C|t||z|^{\alpha}
	\end{equation}
	and 
	\begin{equation}
	\label{Cauchy1}
	\norm{\frac{\partial}{\partial \la}\left[\left(F_t^\D\right)^{-1}F_t\right]}_{\Aa_1} \leq C|t||z|^{\alpha}.
	\end{equation}
\end{lemma}

\begin{proof}
	The first step is to estimate the norm of the positive part $B_t^\D$ of $\Phi_t^\D$.
	We first estimate $\Phi_t^\D$ for $|z|<1$: noting that $A_t$ is diagonalisable, that its eigenvalues tend to $\pm 1/2$ as $t\to 0$, and recalling that $M_t$ is continuous at $t=0$ ensure that for all $\alpha<1$ there exists $(T,R)$ and $C_1>1$ such that for all $|t|<T$,
	\begin{equation*}
	\norm{\Phi_t^\D(z,\la)}_{\Aa_R} \leq C_1|z|^{-\frac{1}{2}-\frac{1-\alpha}{4}}.
	\end{equation*}
	We then estimate $F_t^\D$: let $\gamma\subset \Cc^*$ be a path from $z$ to $1$, use Equation \eqref{eqestimeeFannexe} of Appendix \ref{appendixF} and Equation \eqref{eqmetric} to get
	\begin{equation*}
	\norm{F_t^\D(z,\la)}_{\Aa_R} \leq C_2 \norm{F_t^\D(1,\la)}_{\Aa_R} \times \exp\left( \frac{(R-1)}{2} \int_{\gamma} \frac{|\sigma_t(\log |z|)|}{|z|} \right).
	\end{equation*}
	But $\sigma_t$ is uniformly bounded because so is the distance between the profile curve and the axis of a Delaunay surface. Moreover, the unitary frame at $z=1$ is also bounded. Hence the existence, for $R>1$ small enough, of a constant $C_3\geq 1$ such that
	\begin{equation*}
	\norm{F_t^\D(z,\la)}_{\Aa_R}\leq C_3 |z|^{-\frac{1-\alpha}{4}}.
	\end{equation*}
	We can now estimate the positive factor: for all $\alpha<1$ there exist $T>0$, $R>1$ and $C_4\geq 1$ such that for all $|t|<T$ and $|z|<1$
	\begin{equation*}
	\norm{B_t^\D(z,\la)}_{\Aa_R} \leq \norm{F_t^\D(z,\la)^{-1}}_{\Aa_R} \times \norm{\Phi_t^\D(z,\la)}_{\Aa_R}\leq C_4|z|^{\frac{\alpha}{2}-1}.
	\end{equation*}

	We then define
	\begin{equation*}
	\begin{array}{rccccc}
	\widetilde{\Phi}_t &:=& \left(\left(F_t^\D \right)^{-1}F_t\right)&\times&\left(B_t\left(B_t^\D\right)^{-1}\right) &= B_t^\D \left(\Phi_t^\D\right)^{-1}\Phi_t \left(B_t^\D\right)^{-1} \\
	&=:& \widetilde{F}_t&\times & \widetilde{B}_t &
	\end{array}
	\end{equation*}
	with $\widetilde{F}_t\in\LSU$ and $\widetilde{B}_t\in\LSLplusR$ and  thus have
	\begin{align*}
	\norm{\widetilde{\Phi}_t(z,\la) - \I}_{\Aa_R} &= \norm{B_t^\D(z,\la)\left( P_t(z,\la)-\I \right) \left(B_t^\D(z,\la)\right)^{-1}}_{\Aa_R}\\
	&\leq \norm{B_t^\D(z,\la)}_{\Aa_R}^2 \Oo(t,|z|^2)\\
	&\leq C|t||z|^\alpha.
	\end{align*}
	Let $n_k$ denote the seminorms
	\begin{equation*}
	\label{defseminorm}
	n_k(X) = \sum_{j=0}^{k} \norm{\frac{\partial^kX}{\partial\la^k}}_{\Aa_1}.
	\end{equation*}
	Apply Cauchy formula with $\la\in\partial\Aa_R$ to get
	\begin{equation*}
		n_k\left(\widetilde{\Phi}_t - \I\right)\leq c_k |t||z|^{\alpha}, \ \forall k\in\Nn
	\end{equation*}
	where $c_k>0$ are  uniform constants.
	But $\Uni(\widetilde{\Phi}_t) = \widetilde{F}_t = \left(F_t^\D\right)^{-1}F_t$ and Iwasawa decomposition is a $\C^1$-diffeomorphism, so $n_0\left(\widetilde{F}_t-\I\right)\leq C|t||z|^{\alpha}$ and $n_1\left(\widetilde{F}_t-\I\right)\leq C|t||z|^{\alpha}$. We then have \eqref{Cauchy0} and \eqref{Cauchy1}.
\end{proof}

The asymptotic behaviour of $\frac{\partial\widetilde{F}_t}{\partial\la}$ allows us to prove the convergence of immersions as stated in the first point of Theorem \ref{theorem}. The Sym-Bobenko formula for $\Rr^3$ implies that (we omit the index $t$)
\begin{align*}
i F(z,1) \frac{\partial (F^{-1}F^\D)}{\partial \la}(z,1)F^\D(z,1)^{-1}
&= i \frac{\partial F^\D}{\partial \la}(z,1) F^\D(z,1)^{-1} - i \frac{\partial F}{\partial \la}(z,1) F(z,1)^{-1} \\
&= f^\D(z) - f(z).
\end{align*}
We can then compute
\begin{align*}
\norm{f_t(z)-f_t^\D(z)}_{\Rr^3}^2 &= 4\det\left(f_t(z)-f_t^\D(z)\right)\\
&=-4\det \frac{\partial (F_t^{-1}F_t^\D)}{\partial \la}(z,1)\\
& \leq C_2^2t^2|z|^{2\alpha}.
\end{align*}
And then for all $\alpha<1$ there exist constants $\epsilon>0$, $T>0$ and $C>0$ such that for all $0<|z|<\epsilon$ and $|t|<T$,
\begin{equation}\label{eqconvimmersions}
\Vert f_t(z) - f_t^{\D}(z)\Vert _{\Rr^3} \leq C|t||z|^{\alpha}.
\end{equation}

To prove the third point of Theorem \ref{theorem}, use \eqref{PhitD} and note that $M_0=\I$. So the axis of $f_t^\D$ as $t\to 0$ is the same that the axis of the unperturbed Delaunay surface induced by $z^{A_t}$.

In order to prove that the surface is embedded, we will need the convergence of the normal maps:

\begin{proposition}\label{corconvnormales}
	For all $\alpha<1$ there exist constants $\epsilon>0$, $T>0$ and $C>0$ such that for all $0<|z|<\epsilon$ and $|t|<T$,
	\begin{equation*}
	\norm{\N_t(z)-\N_t^\D(z)}_{\Rr^3} \leq C|t||z|^{\alpha}
	\end{equation*}
\end{proposition}
\begin{proof}
	Use the definition of the normal maps in Equation \eqref{defN} to write
	\begin{align*}
	\N_t(z)-\N_t^\D(z) = \frac{-i}{2}F_t^\D(z,1)\left[ AM\widetilde{A} + AM + M\widetilde{A} \right]F_t^\D(z,1)^{-1}
	\end{align*}
	where
	\begin{equation*}
		A=F_t^\D(z,1)^{-1}F_t(z,1)-\I = \Oo(t,|z|^\alpha),
	\end{equation*}
	\begin{equation*}
		\widetilde{A} = F_t(z,1)^{-1}F_t^\D(z,1) - \I = \Oo(t,|z|^\alpha)
	\end{equation*}
	and
	\begin{equation*}
		M=\begin{pmatrix}
		1 & 0\\
		0 & -1
		\end{pmatrix}.
	\end{equation*}
	Use equation \eqref{eqnormR3} to get the conclusion.
\end{proof}

It remains to show that the surface is embedded if $t>0$.

\section{Embeddedness}
\label{sectionembeddedness}

We suppose in this section that $0<t<T$. The asymptotic behaviour of $f_t$ and the fact that $f_t^\D$ is an embedding for all $t$ allow us to show that $f_t$ is an embedding of a sufficiently small uniform neighbourhood of $z=0$ for $t$ small enough. We first give a general result of embeddedness and then apply this result to show that our surfaces are embedded.

\begin{proposition}
	\label{propgenericresultembeddedness}
	Let $f_n^\R : \Cc^* \longrightarrow \M_n^\R = f_n^\R(\Cc^*) \subset \Rr^3$ be a sequence of complete immersions with normal maps $\N_n^\R$ and an end at $z=0$. Suppose that for all $n$ there exists $r_n>0$ such that the tubular neighbourhood $\Tub_{r_n}\M_n^\R$ of $\M_n^\R$ is embedded. Suppose that for all $\epsilon>0$ there exists $0<\epsilon'<\epsilon$ such that for all $n\in\Nn$, $x\in\Ss_\epsilon$ and $y\in\Dd_{\epsilon'}^*$,
	\begin{equation}
	\label{eqfnR-fnR}
	\normR{f_n^\R(x)-f_n^\R(y)}>2r_n.
	\end{equation}
	Let $U^*\subset \Cc^*$ be a punctured neighbourhood of $z=0$ and $f_n:U^*\longrightarrow \Rr^3$ a sequence of immersions with normal maps $\N_n$ satisfying
	\begin{equation}
	\label{Hyp1}
	\supp{n\in\Nn} \ \frac{\norm{f_n(z) - f_n^\R(z)}_{\Rr^3}}{r_n} \tendsto{z\to 0} 0 
	\end{equation}
	and
	\begin{equation}
	\label{Hyp2}
	\supp{z\in U^*}\norm{\N_n(z) - \N_n^\R(z)}_{\Rr^3} \tendsto{n\to\infty}0.
	\end{equation}
	Then there exist $\epsilon'>0$ and $N\in\Nn$ such that for all $n\geq N$, $f_n$ is an embedding of $\Dd_{\epsilon'}^*$.
\end{proposition}

\begin{proof}
	Let us split the proof in several steps.
	\begin{itemize}
		\item \textit{Claim $1$}: there exists $\epsilon>0$ such that the map
		\begin{align*}
		\begin{array}{ccccc}
		\varphi_n & : & \Dd_\epsilon^* & \longrightarrow & \M_n^\R \\
		& & z & \longmapsto & \pi_n\circ f_n(z) \\
		\end{array}
		\end{align*}
		(where $\pi_n$ is the projection from $\Tub_{r_n} \M_n^\R$ onto $\M_n^\R$) is well-defined and satisfies
		\begin{equation}
		\label{eqphin-fnR}
		\norm{\varphi_n(z) - f_n^\R(z)}_{\Rr^3} < r_n
		\end{equation}
		for all $z\in\Dd_\epsilon^*$.
	\end{itemize}
		
		To prove this first claim, use Hypothesis \eqref{Hyp1}: there exists $\epsilon>0$ such that for all $n\in\Nn$ and $z\in\Dd_\epsilon^*$
		\begin{equation}
		\label{eqfn-fnR}
		\norm{f_n(z) - f_n^\R(z)}_{\Rr^3} < \frac{r_n}{2}.
		\end{equation}
		So $f_n(\Dd_\epsilon^* )\subset \Tub_{\frac{r_n}{2}}\M_n^\R$ and $\varphi_n$ is well-defined. Moreover, using \eqref{eqfn-fnR} and the triangle inequality, for all $z\in\Dd_{\epsilon}^*$
		\begin{equation*}
		\norm{\varphi_n(z) - f_n^\R(z)}_{\Rr^3} \leq \norm{\varphi_n(z) - f_n(z)}_{\Rr^3} + \norm{f_n(z) - f_n^\R(z)}_{\Rr^3} <r_n
		\end{equation*}
		and Equation \eqref{eqphin-fnR} holds.
		We fix $\epsilon$ and $\epsilon'$ so that Equation \eqref{eqfnR-fnR} is satisfied.
		
	\begin{itemize}
		\item \textit{Claim 2}: there exists $N\in\Nn$ such that for all $n\geq N$, $\varphi_n$ is a local diffeomorphism on $\Dd_\epsilon^*$.
	\end{itemize}
	Let $z\in \Dd_\epsilon^*$. In order to show that $\varphi_n$ is a local diffeomorphism, we show that
	\begin{equation}
	\label{eqNphiscalNn}
	\left\langle \N_{\varphi_n}(z), \N_n(z) \right\rangle > 0
	\end{equation}
	where $\N_{\varphi_n}$ is defined by
	\begin{align*}
	\begin{array}{ccccc}
	\N_{\varphi_n} & : & \Dd_\epsilon^* & \longrightarrow & \Ss^2 \subset \Rr^3 \\
	& & z & \longmapsto & \eta_n^\R(\varphi_n(z)) \\
	\end{array}
	\end{align*}
	and $\eta_n^\R$ is the Gauss map of $\M_n^\R$. First, let $\gamma\subset \M_n^\R$ be a path joining $\varphi_n(z)$ to $f_n^\R(z)$. Using the fact that $\Tub_{r_n}\M_n^\R$ is embedded, one has
	\begin{equation*}
		\norm{d\eta_n^\R}\leq \frac{1}{r_n}
	\end{equation*}
	and
	\begin{equation*}
	\left\Vert \N_{\varphi_n}(z) - \N_n^\R(z) \right\Vert_{\Rr^3} \leq \frac{1}{r_n} \times |\gamma|.
	\end{equation*}
	Let $\sigma(t) = (1-t)f_n(z) + tf_n^\R(z)$, $t\in\left[0,1\right]$. Then,
	\begin{equation}\label{eqsigmatub}
	\norm{\sigma(t) - f_n^\R(z)}_{\Rr^3} \leq (1-t) \norm{f_n(z) - f_n^\R(z)}_{\Rr^3} < \frac{r_n}{2}
	\end{equation}
	because of Equation \eqref{eqfn-fnR}. Let $\gamma = \pi_n\circ\sigma$.
	Note that Equation \eqref{eqsigmatub} implies that $\sigma \subset \Tub_{\frac{r_n}{2}}\M_n^\R$ and restricting $\pi_n$ to $\Tub_{\frac{r_n}{2}}\M_n^\R$ gives
	\begin{equation*}
		\norm{d\pi_n}\leq \frac{r_n}{r_n-\frac{r_n}{2}}=2
	\end{equation*}
	and thus $|\gamma|<r_n$. Hence,
	\begin{equation*}
		\norm{\N_{\varphi_n}(z) - \N_n^\R(z)} < 1.
	\end{equation*}
	Use Hypothesis \eqref{Hyp2} to choose a uniform $N\in\Nn$ such that for all $n\geq N$, 
	\begin{equation*}
	\norm{\N_{\varphi_n}(z) - \N_n(z)} \leq \norm{\N_{\varphi_n}(z) - \N_n^\R(z)} + \norm{ \N_n^\R(z) - \N_n(z)}  < \sqrt{2},
	\end{equation*}
	which proves Equation \eqref{eqNphiscalNn} and this second claim.  We fix such $N$  and $n$.
		
	\begin{itemize}
		\item \textit{Claim 3}: the restriction
		\begin{equation*}
		\function{\widetilde{\varphi}_n}{\varphi_n^{-1}\left(\varphi_n(\Dd_{\epsilon'}^*)\right)\cap \Dd_\epsilon^*}{\varphi_n\left( \Dd_{\epsilon'}^* \right)}{z}{\varphi_n(z)}
		\end{equation*}
		is a covering map.
	\end{itemize}
	
	It sufices to show that $\widetilde{\varphi}_n$ is a proper map. Let $(x_i)_{i\in\Nn}\subset \varphi_n^{-1}\left(\varphi_n(\Dd_{\epsilon'}^*)\right)\cap \Dd_\epsilon^*$ such that $\left(\widetilde{\varphi}_n(x_i)\right)_{i\in\Nn}$ converges to $p\in\varphi_n\left(\Dd_{\epsilon'}^*\right)$. Then $(x_i)_i$ converges to $x\in\overline{\Dd}_\epsilon$. Using Equation \eqref{eqphin-fnR} and the fact that $f_n^\R$ has an end at $0$, $x\neq 0$. If $x\in\partial\Dd_\epsilon$, denoting $\widetilde{x}\in\Dd_{\epsilon'}^*$ such that $\widetilde{\varphi}_n(\widetilde{x}) = p$, one has
	\begin{equation*}
	\normR{f_n^\R(x) - f_n^\R(\widetilde{x})} < \normR{f_n^\R(x) - p} + \normR{f_n^\R(\widetilde{x})-\widetilde{\varphi}_n(\widetilde{x})} < 2r_n
	\end{equation*}
	which contradicts the definition of $\epsilon'$. Thus, $\widetilde{\varphi}_n$ is a proper local diffeomorphism between locally compact spaces, i.e. a covering map.

	\begin{itemize}
		\item \textit{Claim 4}: this covering map is one-sheeted.
	\end{itemize}
	
	To compute the number of sheets, let $\gamma : [0,1] \longrightarrow \Dd_{\epsilon'}^*$ be a loop of winding number $1$ around $0$, $\Gamma = f_n^\R(\gamma)$ and $\widetilde{\Gamma} = \widetilde{\varphi}_n(\gamma) \subset\M_n^\R$ and let us construct a homotopy between $\Gamma$ and $\widetilde{\Gamma}$. Let
	\begin{equation*}
	\function{\sigma_t}{[0,1]}{\Rr^3}{s}{(1-s)\widetilde{\Gamma}(t) + s \Gamma(t).}
	\end{equation*}
	For all $t,s\in[0,1]$,
	\begin{align*}
	\normR{\sigma_t(s) - \Gamma(t)} < r_n
	\end{align*}
	which implies that $\sigma_t(s) \in \Tub_{r_n}\M_n^\R$ because $\M_n^\R$ is complete. One can thus define the following homotopy between $\Gamma$ and $\widetilde{\Gamma}$
	\begin{equation*}
	\function{H}{[0,1]^2}{\M_n^\R}{(s,t)}{\pi_n\circ \sigma_t(s)}
	\end{equation*}
	where $\pi_n$ is the projection from $\Tub_{r_n}\M_n^\R$ to $\M_n^\R$. Using the fact that $f_n^\R$ is an embedding, the degree of $\Gamma$ is one, and the degree of $\widetilde{\Gamma}$ is also one. Hence, $\widetilde{\varphi}_n$ is one-sheeted.
	
	\begin{itemize}
		\item \textit{Conclusion}: the map $\widetilde{\varphi}$ is a diffeomorphism, so $f_n\left(\Dd_{\epsilon'}^*\right)$ is a graph over $\M_n^\R$ contained in its embedded tubular neighbourhood and $f_n\left(\Dd_{\epsilon'}^*\right)$ is thus embedded.
	\end{itemize}
\end{proof}

We can now apply Proposition \ref{propgenericresultembeddedness} to each case. Let $(t_n)$ be any sequence in $(-T,T)$ such that $t_n \to 0$.
\begin{itemize}
	\item If $r\geq s$, we set $\widehat{f}_n^\R = f_{t_n}^\D$ and $\widehat{f}_n = f_{t_n}$. We aim to apply Proposition \ref{propgenericresultembeddedness} on $\widehat{f}_n^\R$ and $\widehat{f}_n$. The tubular radius $r_n$ is of the order of $4 t_n$ and Hypothesis \eqref{eqfnR-fnR} is satisfied because $\widehat{f}_n^\R$ tends to an immersion of a sphere. Equation \eqref{eqconvimmersions} and Proposition \ref{corconvnormales} ensure that Hypotheses \eqref{Hyp1} and \eqref{Hyp2} hold.
	
	\item If $r\leq s$, we set $\widehat{f}_n^\R = \frac{1}{t_n} f_{t_n}^\D$ and $\widehat{f}_n =\frac{1}{t_n} f_{t_n}$. We aim to apply Proposition \ref{propgenericresultembeddedness} on $\widehat{f}_n^\R$ and $\widehat{f}_n$. The tubular radius $r_n$ is of the order of $4$ and Hypothesis \eqref{eqfnR-fnR} is satisfied because $\widehat{f}_n^\R$ tends to an immersion of a catenoid (see \cite{minoids}). Equation \eqref{eqconvimmersions} and Proposition \ref{corconvnormales} ensure that Hypotheses \eqref{Hyp1} and \eqref{Hyp2} hold.
\end{itemize}
The second point of our theorem is then proved.
\appendix

\section{Iwasawa extended}\label{SectionIwaExtended}

In this section, we note $\Aa_{\frac{1}{R},1} = \left\{ \la\in\Cc \ : \ \frac{1}{R}<|\la|<1 \right\}$.

\begin{lemma}\label{lemmaFanneau}
	Let $F:\Aa_{\frac{1}{R},1}\longrightarrow \SL_2\Cc$ be a holomorphic map that can be continuously extended to the circle $\Aa_1$ and such that $F(\la)\in\SU_2$ for all $\la\in\Aa_1$. Then $F$ holomorphically extends to $\Aa_R$ into a map that satisfies
	\begin{equation}\label{eqprolongementF}
	^t{\overline{F\left(\frac{1}{\conj{\la}}\right)}} = F(\la)^{-1} \qquad \forall \la\in\Aa_R.
	\end{equation}
\end{lemma}
\begin{proof}
	Apply Schwarz reflexion principle on each coefficient of the matrix
	\begin{equation*}
	\widetilde{F}(\la) = \begin{pmatrix}
	F_{11}(\la) + F_{22}(\la)& F_{12}(\la) - F_{21}(\la)\\
	i\left( F_{12}(\la) + F_{21}(\la) \right) & i\left( F_{11}(\la) - F_{22}(\la) \right)
	\end{pmatrix}
	\end{equation*}
	where $F_{ij}$ denote the entries of $F$. The fact that $F(\la)\in\SU_2$ for all $\la\in\Aa_1$ ensures that $\Im \widetilde{F} = 0$ on $\Aa_1$. Thus, $\widetilde{F}$ holomorphically extends to $\Aa_R$ and satisfies for all $\la\in\Aa_R$ 
	\begin{equation*}
	{\widetilde{F}\left( \frac{1}{\conj{\la}} \right)} = \conj{\widetilde{F}(\la)}.
	\end{equation*}
	Hence, $F$ holomorphically extends to $\Aa_R$ and satisfies
	\begin{equation*}\label{eqcoefsF}
	F_{11}\left(\frac{1}{\conj{\la}}\right) = \conj{F_{22}(\la)}, \qquad F_{12}\left(\frac{1}{\conj{\la}}\right) = -\conj{F_{21}(\la)}
	\end{equation*}
	which implies Equation \eqref{eqprolongementF} because $F(\la)\in\SL_2\Cc$.
\end{proof}

\begin{corollary}\label{IwasawaExtended}
	Let $\Phi:\Aa_R\longrightarrow\SL_2\Cc$ be a holomorphic map and let $FB$ be the Iwasawa decomposition of its restriction to $\Aa_1$. Then $F$ holomorphically extends to $\Aa_R$, satisfies Equation \eqref{eqprolongementF}, and $B$ holomorphically extends to $\D_R$.
\end{corollary}
\begin{proof}
	Write $F=\Phi B^{-1}$ to holomorphically extend $F$ to $\Aa_{\frac{1}{R},1}$. Apply Lemma \ref{lemmaFanneau} to holomorphically extend $F$ to $\Aa_R$, and write $B=F^{-1}\Phi$ to holomorphically extend $B$ to $\D_R$.
\end{proof}

\section{Derivative of the monodromy}
\label{appendixderiveemonod}

The following proposition, used in Section \ref{sectionzap}, is derived from Proposition 8 in \cite{nnoids}.

\begin{proposition}
	\label{propderiveemonod}
	Let $\xi_t$ be a $\C^1$ family of matrix-valued $1$-forms on a Riemann surface $\Sigma$, defined for $t$ in a neighbourhood of $t_0\in\Rr$. Let $\widetilde{\Sigma}$ be the universal cover of $\Sigma$. Fix a point $z_0$ in $\Sigma$ and let $\widetilde{z}_0$ be a lift of $z_0$ to $\widetilde{\Sigma}$. Let $\Phi_t$ be a continuous family of solutions of $d\Phi_t = \Phi_t\xi_t$ on $\widetilde{\Sigma}$ such that for all $t$,
	\begin{equation*}
		\label{eqMonodcommute}
		\left[ \M(t_0), \Phi_{t_0}(z_0)\Phi_t(z_0)^{-1} \right] = 0,
	\end{equation*}
	where $\M(t)$ is the monodromy of $\Phi_t$ with respect to some $\gamma\in \pi_1(\Sigma,z_0)$. Let $\widetilde{\gamma}$ be the lift of $\gamma$ to $\widetilde{\Sigma}$ such that $\widetilde{\gamma}(0)=\widetilde{z}_0$. Then $\M$ is differentiable at $t_0$ and 
	\begin{equation*}
	\M'(t_0) = \left( \int_{\gamma} \Phi_{t_0} \frac{\partial \xi_t}{\partial t}\mid_{\substack{t=t_0}} \Phi_{t_0}^{-1} \right) \times \M(t_0).
	\end{equation*}
	In particular, if $\M(t_0)=\pm \I$ or if $\Phi_t(z_0)$ is constant, then \eqref{eqMonodcommute} is satisfied.
\end{proposition}
\begin{proof}
	Proposition 8 in \cite{nnoids} is proved in the case where $\Phi_t(z_0)$ is constant. Let $\widetilde{\Phi}_t(z) = \Phi_t(z_0)^{-1}\Phi_t(z)$, so that $d\widetilde{\Phi}_t = \widetilde{\Phi}_t\xi_t$ and $\widetilde{\Phi}_t(z_0) = \mathrm{I}_n$. Let $\widetilde{\M}(t)$ be the monodromy of $\widetilde{\Phi}_t$ along $\gamma$. Then Proposition 5 of \cite{nnoids} applies and 
	\begin{equation*}
	\widetilde{\M}'(t_0) = \left( \int_{\gamma} \widetilde{\Phi}_{t_0}(z) \frac{\partial \xi_t(z)}{\partial t}\mid_{\substack{t=t_0}} \widetilde{\Phi}_{t_0}(z)^{-1} \right) \times \widetilde{\M}(t_0).
	\end{equation*}
	On the other hand, 
	\begin{equation*}
	\M(t) = \Phi_t(z_0)\widetilde{\M}(t) \Phi_t(z_0)^{-1}
	\end{equation*}
	and because of Equation \eqref{eqMonodcommute},
	\begin{equation*}
	\M(t_0) = \Phi_t(z_0)\widetilde{\M}(t_0) \Phi_t(z_0)^{-1}.
	\end{equation*}
	Thus, $\M$ is differentiable at $t_0$ and 
	\begin{equation*}
	\M'(t_0) = \Phi_{t_0}(z_0)\widetilde{\M}'(t_0)\Phi_{t_0}(z_0)^{-1}
	\end{equation*}
	which proves the proposition.
\end{proof}

\section{A control formula on the unitary frame}
\label{appendixF}

The following proposition is used in Section \ref{sectionconvergence}.
\begin{proposition}
	Let $\left( \Sigma, \xi, z_0, \Phi_{z_0} \right)$ be a set of untwisted DPW data, holomorphic for $\la\in\Aa_R$ with $R\geq 1$. Then for all $z_1,z_2\in\Sigma$ and $\gamma\subset \Sigma$ joining $z_1$ to $z_2$,
	\begin{equation*}
	\norm{F(z_1,\la)}_{\Aa_R} \leq  C \norm{F(z_2,\la)}_{\Aa_R} \times \exp\left( (R-1) \int_{\gamma} \rho^2(w) |a_{-1}(w)||dw| \right)
	\end{equation*}
	where $C$ is a uniform positive constant, $a_{-1}(z)dz$ is the $\la^{-1}$ factor of $\xi$ and $\rho(z)$ is the upper-left entry of $\Pos(\Phi)(z,0)$.
\end{proposition}
\begin{proof}
	Write
	\begin{equation*}
	\xi(z,\la) = \la^{-1}\begin{pmatrix}
	0 & a_{-1}(z)\\
	0 & 0
	\end{pmatrix} dz + \la^{0}\begin{pmatrix}
	c_0(z) & a_0(z)\\
	b_0(z) & -c_0(z)
	\end{pmatrix}dz + \Oo(\la).
	\end{equation*}
	Let $\Phi=FB$ be the Iwasawa decomposition of $\Phi$. Untwisting formula (4.3.5) of \cite{loopgroups} with the help of Remark 4.2.6 of \cite{loopgroups} gives $dF=FL$ where
	\begin{equation*}
	L(z,\la) = \begin{pmatrix}
	\rho^{-1}\rho_z & \la^{-1}\rho^2a_{-1}\\
	b_0\rho^{-2} & -\rho^{-1}\rho_z
	\end{pmatrix}dz + \begin{pmatrix}
	-\rho^{-1}\rho_{\bar{z}} & -\conj{b}_0\rho^{-2}\\
	-\la\rho^2\conj{a}_{-1} & \rho^{-1}\rho_{\bar{z}}
	\end{pmatrix}d\bar{z}.
	\end{equation*}
	Let
	\begin{equation*}
	\widetilde{F}(z,\la) = F\left( z,\frac{\la}{|\la|} \right)
	\end{equation*}
	so that $\widetilde{F}(z,\la)\in\SU_2$ for all $\la\in\Aa_R$. Then $d\widetilde{F} = \widetilde{F}\widetilde{L}$ where
	\begin{equation*}
	\widetilde{L}(z,\la) = L\left( z,\frac{\la}{|\la|} \right).
	\end{equation*}
	Using the variation of constants method, for all $z_1,z_2\in\Sigma$ (we ommit the variable $\la$),
	\begin{equation*}
	F(z_1) = F(z_2)\widetilde{F}(z_2)^{-1}\widetilde{F}(z_1) + \left( \int_{z_2}^{z_1} F(w)\left( L(w)-\widetilde{L}(w) \right) \widetilde{F}(w)^{-1} \right)\widetilde{F}(z_1).
	\end{equation*}
	But
	\begin{equation*}
	L(w,\la) - \widetilde{L}(w,\la) = \rho^2(w)\begin{pmatrix}
	0 & a_{-1}(w)\la^{-1}\left( 1-|\la| \right)dw\\
	-\conj{a}_{-1}(w)\la\left( 1-|\la|^{-1} \right)d\bar{w} & 0
	\end{pmatrix}
	\end{equation*}
	so there exists a uniform constant $\widetilde{C}$ such that
	\begin{equation*}
	\norm{L(w,\la) - \widetilde{L}(w,\la)}_{\Aa_R} \leq \widetilde{C} (R-1) \rho^2(w) |a_{-1}(w)| |dw|
	\end{equation*}
	and the result follows from Gronwall's inequality (Lemma 2.7 in \cite{teschl}) using the fact that $\widetilde{F}\in\SU_2$ for all $\la\in\Aa_R$.
\end{proof}

As an application, recall that in the untwisted $\Rr^3$ setting, if $f=\Sym(F)$, then $f$ is a CMC 1 conformal immersion whose metric is given by
\begin{equation*}
ds = 2\rho^2|a_{-1}||dz|.
\end{equation*}
So let $z_1,z_2\in\Sigma$ and $\gamma\subset\Sigma$ be a path joining $f(z_1)$ to $f(z_2)$. Then,
\begin{equation}\label{eqestimeeFannexe}
\norm{F(z_1,\la)}_{\Aa_R} \leq C \norm{F(z_2,\la)}_{\Aa_R} \exp\left( \frac{(R-1)}{2} |\gamma| \right).
\end{equation}

\pagebreak

%

\providecommand{\bysame}{\leavevmode\hbox to3em{\hrulefill}\thinspace}
\providecommand{\MR}{\relax\ifhmode\unskip\space\fi MR }
\providecommand{\MRhref}[2]{%
	\href{http://www.ams.org/mathscinet-getitem?mr=#1}{#2}
}
\providecommand{\href}[2]{#2}

\noindent
Thomas Raujouan\\
Institut Denis Poisson\\
Universit\'e de Tours, 37200 Tours, France\\
\verb$thomas.raujouan@univ-tours.fr$

\end{document}